\documentclass{article}

\usepackage[T2A]{fontenc}
\usepackage[utf8]{inputenc}
\usepackage[russian,english]{babel}
\usepackage{amsfonts}
\usepackage{amsthm}
\usepackage{amsmath}
\usepackage{amssymb}
\usepackage{cancel}
\usepackage{tabu}
\usepackage{array}
\usepackage{tikz}
\usepackage{makecell}
\usepackage{xcolor}
\usepackage{graphicx}
\usepackage[final]{pdfpages}
\usepackage{bbm}

\usepackage{titlesec}
\usepackage{indentfirst}
\usepackage{enumitem}
\usepackage{caption}
\usepackage{refcount}
\usepackage{cite}
\usepackage[colorlinks=true, linkcolor=blue, citecolor=blue, urlcolor=blue, hypertexnames=false]{hyperref}
\usepackage{stmaryrd,scalerel}
\usepackage{titling}
\usepackage{etoc}

\titleformat{\section}[block]          
  {\centering\large\bfseries}          
  {\S\,\thesection.}                   
  {1em}                                
  {}                                   

\makeatletter
\renewenvironment{proof}[1][\proofname]{\par
  \pushQED{\qed}%
  \normalfont \topsep6\p@\@plus6\p@\relax
  \trivlist
  \item[\hspace*{\parindent}\hskip\labelsep \bfseries #1\@addpunct{.}]%
  \ignorespaces
}{%
  \popQED\endtrivlist\@endpefalse
}
\makeatother

\graphicspath{ {./images/} }

\renewcommand{\ge}{\geqslant}

\renewcommand{\le}{\leqslant}

\newtheoremstyle{style2}
  {\topsep}
  {\topsep}
  {\itshape}
  {\parindent}
  {\bfseries}
  {.}
  {5pt plus 1pt minus 1pt}
  {}

\title{На скольки словах Фибоначчи сосредоточена мера Планшереля}

\author{В.~Ю.~Евтушевский}

\begin{document}

\selectlanguage{russian}

\maketitle

\etocsettagdepth{ru}{all}
\etocsettagdepth{en}{none}
\tableofcontents

\theoremstyle{style2}
\newtheorem{Lemma}{Лемма}

\theoremstyle{style2}
\newtheorem{Alg}{Алгоритм}

\theoremstyle{style2}
\newtheorem{Col}{Следствие}

\theoremstyle{style2}
\newtheorem{theorem}{Теорема}

\theoremstyle{style2}
\newtheorem{Def}{Определение}

\theoremstyle{style2}
\newtheorem{Prop}{Предложение}

\theoremstyle{style2}
\newtheorem{Problem}{Задача}

\theoremstyle{style2}
\newtheorem{Zam}{Замечание}

\theoremstyle{style2}
\newtheorem{Oboz}{Обозначение}

\theoremstyle{style2}
\newtheorem{Ex}{Пример}

\theoremstyle{style2}
\newtheorem{Nab}{Наблюдение}

\begin{abstract}
На множестве $\mathbb{YF}_n$ слов из $1$ и $2$ 
с суммой цифр $n$ рассматривается вероятностное распределение
$\mu_P$,
соответствующее мере Планшереля на путях в графе
Юнга --- Фибоначчи. Изучается асимптотика по $n$ энтропии
 распределения $\mu_P$ и минимального количества
слов, на которых $\mu_P$ почти полностью сосредоточено. 
\end{abstract}

\etocdepthtag.toc{ru}
\section{Введение}
Рассмотрим слова над алфавитом $\{1,2\}$ с данной суммой цифр $n.$ Как известно, их количество есть число Фибоначчи $F_{n+1}$ ($F_0=0,$ $F_1=1,$ $F_{k+2}=F_{k+1}+F_k$), и это самая распространённая комбинаторная интерпретация чисел Фибоначчи. Также можно думать о разбиениях полосы $2\times n$ на домино $1\times 2$ и $2\times 1,$ сопоставляя двойки парам горизонтальных домино, а единицы вертикальным домино.

Введём на этом множестве слов частичный порядок: будем говорить, что слово $v$ \emph{предшествует} слову $w$, если после удаления самого длинного общего суффикса в слове $w$ остаётся не меньше двоек, чем в слове $v$ остаётся цифр. 

Это действительно частичный порядок, более того, соответствующее частично упорядоченное множество является модулярной решёткой, известной как решётка Юнга --- Фибоначчи.

Графом Юнга --- Фибоначчи (он изображён на рис. \ref{yf}) называют диаграмму Хассе этой решётки. Это градуированный граф, который мы представляем растущим снизу вверх, начиная с пустого слова. Градуировкой служит функция суммы цифр. Опишем явно, как устроены ориентированные рёбра. Рёбра <<вверх>> из данного слова $v$ ведут в слова, получаемые из $v$ одной из двух операций:
\begin{enumerate}
    \item заменить самую левую единицу на двойку;
    \item вставить единицу в такое место,
    что левее неё нет единиц.
\end{enumerate}

\begin{figure}[htbp]
    \centering
    \includegraphics[width=12cm, height=10cm]{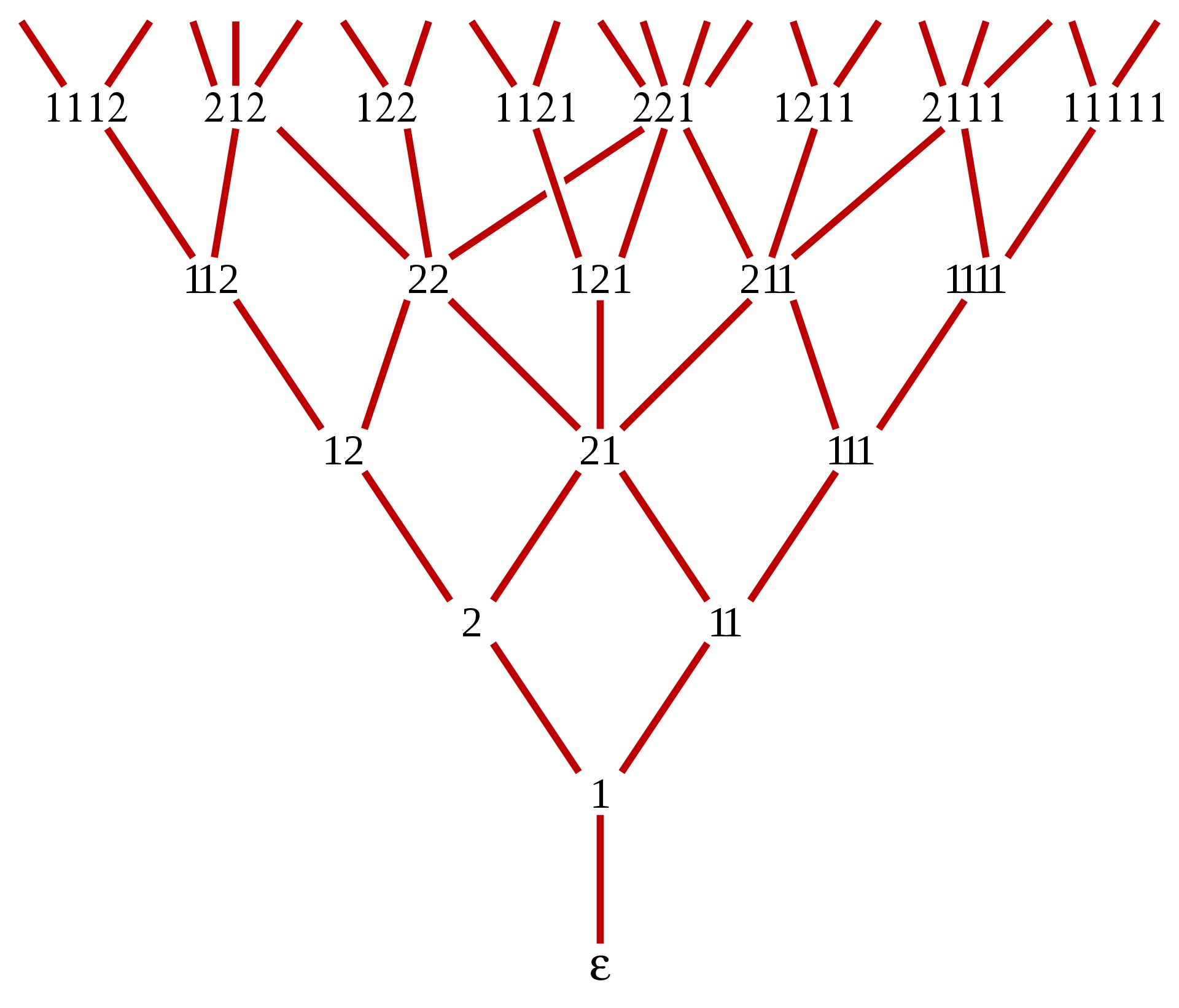}
    \caption{Граф Юнга --- Фибоначчи}
    \label{yf}
\end{figure}

В данной работе будем считать граф Юнга --- Фибоначчи ориентированным и содержащим исключительно рёбра <<вверх>>.

Этот граф помимо модулярности является $1$-дифференциальным, то есть для каждой вершины исходящая степень на $1$ превосходит входящую степень.

Изучение градуированного графа Юнга --- Фибоначчи было инициировано в 1988 г. одновременно и независимо друг от друга такими математиками, как Ричард Стенли \cite{St} и Сергей Владимирович Фомин \cite{Fo}.

Одна из причин интереса к нему в том, что существует всего две корневые локально конечные $1$-дифференциальные модулярные решётки, вторая --- это решётка диаграмм Юнга, имеющая ключевое значение в теории представлений симметрической группы \cite[Theorem~1.3]{Bur}. 

Ключевые вопросы о градуированных графах касаются центральных мер на пространстве (бесконечных) путей в графе. Эта точка зрения последовательно развивалась в работах Анатолия Моисеевича Вершика, к недавнему обзору которого \cite{Ver} и приводимой там литературе мы отсылаем читателя.

Среди центральных мер выделяют те, которые являются пределами мер, индуцированных путями в далёкие вершины --- так называемую \emph{границу Мартина} графа.

Граница пространства путей графа Юнга --- Фибоначчи изучалась в работе Фредерика Гудмана и Сергея Васильевича Керова (2000) \cite{GK}. Они использовали алгебраический формализм Окады \cite{Okada}.

Гудман и Керов доказали, что список интересующих нас центральных мер исчерпывается следующими мерами:
\begin{enumerate}
    \item Мера Планшереля: мера множества путей, проходящих через данную вершину $v$ ранга $n,$ равна ${D(v)^2}/{n!}$, где $D(v)$ --- количество путей <<вверх>> из корня в вершину $v$.
    \item Меры $\mu_{w,\beta},$ параметризующиеся некоторым бесконечным словом $w,$ содержащим <<достаточно мало>> двоек, и числом $\beta\in(0,1]$.
\end{enumerate}

Данная работа посвящена мере Планшереля. Эта мера исследовалась, в частности, С.~В.~Керовым и Александром Васильевичем Гнединым  \cite{KerGned}, которые доказали
её эргодичность. В работе \cite{BE} Иваном Алексеевичем Бочковым 
и автором были получены явные формулы для остальных мер и доказана их 
эргодичность.

Для каждого $n$
мера Планшереля задаёт вероятностную меру на множестве $\mathbb{YF}_n$ 
вершин графа Юнга --- Фибоначчи ранга $n.$ В данной работе
изучается именно эта последовательность вероятностных мер.

В разделе 3 мы вычисляем асимптотику (с точностью до $O(1)$) энтропий таких мер. В разделе 4 мы доказываем, что существует константа $c'\approx 1,79556,$ такая, что для всех $c>c'$ при больших $n$ мера Планшереля \emph{сосредоточена} на $\lfloor\exp{(c\cdot\ln^2{n})}\rfloor$ вершинах графа Юнга --- Фибоначчи ранга $n.$ В разделе 5 мы доказываем, что при больших $n$ мера Планшереля  
\emph{не сосредоточена} ни на каких $\lfloor\exp{({1}/{16}\cdot\ln^2{n})}\rfloor$ вершинах графа Юнга --- Фибоначчи ранга $n.$

Отметим, что энтропия меры Планшереля графа Юнга 
(заданная на множестве $\mathbb{Y}_n$ диаграмм Юнга с $n$
клетками) растёт значительно
быстрее --- как $\Theta(\sqrt{n})$ \cite{VK}, поэтому она не сосредоточена
менее чем на $\exp({\rm const}\sqrt{n})$ вершинах. Это тем
удивительнее, что \emph{общее} число вершин ранга $n$
графа Юнга --- Фибоначчи (то есть число Фибоначчи $F_{n+1}$) растёт, напротив, значительно быстрее, чем общее число вершин ранга $n$ графа Юнга (то есть число разбиений
$p(n)$ числа $n$).

    \section{Основные понятия и обозначения}

Введём основные понятия и обозначения.

Через $\mathbb{N}_0$ и $\mathbb{N}$ обозначаются множества целых неотрицательных и целых положительных (натуральных) чисел, соответственно. 
Для $m,n\in \mathbb{N}_0:$  $m\leqslant n$ положим
\begin{gather*}
	\overline{m,n}:=\{m,m+1,\ldots,n\},\quad \overline{n}:=\overline{0,n}=\{0,1,\ldots,n\}.
\end{gather*}

Пусть $\mathbb{YF}$ --- это граф Юнга --- Фибоначчи, его вершины --- конечные слова из $1$ и $2.$ Далее $\mathbb{YF}$ как множество --- это множество именно вершин графа Юнга --- Фибоначчи. Для $v,w\in\mathbb{YF}$ определено слово $vw$ --- конкатенация $v$ и $w.$ 

Для $v \in \mathbb{YF}$ сумму цифр слова $v$ обозначим через $|v|$. Число $|v|$ --- \emph{ранг} вершины $v$ в графе $\mathbb{YF}$ (считаем, что пустое слово имеет ранг $0$ и переход по ребру наверх увеличивает ранг на $1$). 

Для целого неотрицательного $n$ множество вершин графа $\mathbb{YF}$ ранга $n$ обозначим через
\begin{gather*}
	\mathbb{YF}_n:=\{v\in\mathbb{YF}:\;|v|=n\}.
\end{gather*}	
В данной работе будем считать граф $\mathbb{YF}$ ориентированным с рёбрами из вершин меньшего ранга в вершины большего ранга.

Для $v\in \mathbb{YF}$ количество путей из $\varepsilon$ в $v$ обозначим через $D(v).$ Хорошо известно, что 
\begin{gather}\label{fd}
    D(\varepsilon)=1,\qquad D(1v)= D(v),\qquad D(2v)=(|v|+1)\cdot D(v).   
\end{gather}
Доказательство см. напр. в статье Гнедина и Керова \cite[формула~3.4]{KerGned}. 

В частности, непосредственно из рекурсивной формулы для $D(v)$ следует, что для любых $n\in\mathbb{N}_0$ 
\begin{gather}\label{factorial}
    \sum_{v\in\mathbb{YF}_n} D^2(v)=n!.
\end{gather}

\emph{Мера Планшереля} $\mu_P$ графа Юнга --- Фибоначчи --- это такая вероятностная мера на пространстве бесконечных путей из корня, что для любого конечного пути $t=(\varepsilon,\ldots,v),$ где $v\in\mathbb{YF}_n,$ мера множества всех бесконечных путей, начинающихся с $t,$ равна $D(v)/n!.$ Через $\mu_P(v)=D^2(v)/n!$ обозначим меру всех путей, проходящих через $v.$

Таким образом, при данном $n\in \mathbb{N}_0$ функция $\mu_P(v)$ задаёт вероятностную меру на множестве $\mathbb{YF}_n.$ В этом смысле понимаются выражения типа $\mu_P(A)$ для подмножеств $A\subseteq\mathbb{YF}_n$.

\emph{Энтропию} такой вероятностной меры на множестве $\mathbb{YF}_n$ обозначим через $H_n.$ То есть
\begin{gather*}
    H_n:=-\sum_{v\in\mathbb{YF}_n}\mu_P(v)\cdot \ln{\mu_P(v)}= \sum_{v\in\mathbb{YF}_n}\frac{D^2(v)}{n!}\cdot(\ln{n!}-2\ln{D(v)}).
\end{gather*}

Также пусть $\{f(n)\}_{n=1}^\infty$ --- последовательность неотрицательных чисел. Тогда при фиксированном $n$ обозначим через $A_{f}(n)$ множество из $\lfloor f(n)\rfloor$ вершин графа Юнга --- Фибоначчи ранга $n$ с наибольшим значением меры Планшереля (вообще говоря, множество $A_f(n)$ определено не однозначно; в случае неоднозначности фиксируем выбор произвольно). Если $\lfloor f(n)\rfloor > |\mathbb{YF}_n|,$ то будем считать, что $A_{f}(n)=\mathbb{YF}_n.$ Кроме того, пусть $\overline{A_{f}}(n)=\mathbb{YF}_n\setminus A_{f}(n).$ Также назовём последовательность $f$ \emph{покрывающей}, если $\lim_{n\to\infty}\mu_{P}(A_f(n))=1$ (или, что то же самое, $\lim_{n\to\infty}\mu_{P}(\overline{A_f}(n))=0)$ и \emph{не покрывающей}, если $\varlimsup_{n\to\infty}\mu_{P}(A_f(n))<1.$

\section{Энтропия последовательности мер}

Сразу докажем рекурсивную формулу для энтропий вероятностных мер, заданных мерой Планшереля на множествах $\mathbb{YF}_n:$ 
\begin{Prop}\label{entrec}
$H_0=H_1=0,$ при $n\ge 2$ выполняется
\begin{gather*}\label{entrec1}
H_n=\frac{1}{n}\cdot H_{n-1}+\frac{n-1}{n}\cdot H_{n-2}+\ln{n}-\frac{n-1}{n}\cdot \ln{(n-1)}.    
\end{gather*}    
\end{Prop}
\begin{proof}
Равенство $H_0=H_1=0$ понятно. При $n\ge 2$ воспользуемся \eqref{fd}:
\begin{align*}
H_n=&\sum_{v\in\mathbb{YF}_{n-1}}\frac{D^2(1v)}{n!}\cdot(\ln{n!}-2\ln{D(1v)})+\sum_{v\in\mathbb{YF}_{n-2}}\frac{D^2(2v)}{n!}\cdot(\ln{n!}-2\ln{D(2v)})\\
=&\frac{1}{n}\cdot\sum_{v\in\mathbb{YF}_{n-1}}\frac{D^2(v)}{(n-1)!}\cdot(\ln{(n-1)!}-2\ln{D(v)}+\ln{n})\\
&+\frac{n-1}{n}\cdot\sum_{v\in\mathbb{YF}_{n-2}}\frac{D^2(v)}{(n-2)!}\cdot(\ln{(n-2)!}-2\ln{D(v)}+\ln{n}-\ln{(n-1)}).
\end{align*}
Теперь предложение следует из \eqref{factorial}.
\end{proof}

Из разложения логарифма $\ln{(1-c/n)}$ в ряд Меркатора непосредственно следует, что при $n\ge 3$ выполняются 
\begin{align}
    \ln{(n-1)}=\ln{n}-\frac{1}{n}+O\left(\frac{1}{n^2}\right),&\qquad \ln{(n-2)}=\ln{n}-\frac{2}{n}+O\left(\frac{1}{n^2}\right);\notag\\
    \ln^2{(n-1)}=\ln^2{n}-\frac{2\ln{n}}{n}+O\left(\frac{\ln{n}}{n^2}\right),&\qquad \ln^2{(n-2)}=\ln^2{n}-\frac{4\ln{n}}{n}+O\left(\frac{\ln{n}}{n^2}\right).\label{f1}    
\end{align} 

Мы готовы вычислить с точностью до $O(1)$ асимптотику энтропий вероятностных мер, заданных мерой Планшереля на множествах $\mathbb{YF}_n:$ 
\begin{theorem}\label{t1}
Имеет место асимптотическое соотношение
    \begin{gather*}\label{t11}
    H_n=\frac{1}{4}\ln^2{n}+\frac{1}{2}\ln{n}+O(1).
    \end{gather*}
\end{theorem}
\begin{proof}
Для каждого натурального $n$ обозначим $h_n:=H_n-1/4\ln^2{n}-1/2\ln{n}.$ Из предложения \ref{entrec} и формул \eqref{f1} следует, что выполняется асимптотическое соотношение
\begin{align*}
    |h_n|=&\biggl|\frac{1}{n}\cdot H_{n-1}+\frac{n-1}{n}\cdot H_{n-2}+\ln{n}-\frac{n-1}{n}\cdot\ln{(n-1)}-\left(\frac{1}{4}\ln^2{n}+\frac{1}{2}\ln{n}\right)\biggr|\\
    =&\biggl|\frac{1}{n}\cdot\left(\frac{1}{4}\ln^2{(n-1)}+\frac{1}{2}\ln{(n-1)}+h_{n-1}\right)+\frac{n-1}{n}\cdot\left(\frac{1}{4}\ln^2{(n-2)}+\frac{1}{2}\ln{(n-2)}+h_{n-2}\right)\\
    &-\frac{1}{4}\ln^2{n}+\frac{1}{2}\ln{n}-\frac{n-1}{n}\cdot\ln{(n-1)}\biggr|\\
    \le& O\left(\frac{\ln{n}}{n^2}\right)+\biggl|\frac{1}{n}\cdot h_{n-1}+\frac{n-1}{n}\cdot h_{n-2}\biggr|+\biggl|-\frac{1}{4}\ln^2{n}+\frac{1}{4n}\cdot\ln^2{n}\\
    &+\frac{n-1}{4n}\cdot\left(\ln^2{n}-\frac{4\ln{n}}{n}\right)+\frac{1}{2}\ln{n}-\frac{2n-3}{2n}\cdot\left(\ln{n}-\frac{1}{n}\right)+\frac{n-1}{2n}\cdot\left(\ln{n}-\frac{2}{n}\right)\biggr|\\
    \le& O\left(\frac{\ln{n}}{n^2}\right)+\max{(|h_{n-1}|,|h_{n-2}|)}.
\end{align*}

Теперь теорема следует из сходимости ряда $\sum_{n=1}^\infty \ln{n}/n^2.$
\end{proof}

\begin{Col}\label{s1} 
Для любой последовательности $\{g(n)\}_{n=1}^\infty$ такой, что $g(n)/\ln^2{n}\to \infty,$ последовательность $f(n)=\exp{g(n)}$ является покрывающей.
\end{Col}
\begin{proof}
Зафиксируем натуральное $n.$ Несложно понять, что для каждой вершины $v\in\overline{A_f}(n)$ выполняется $\mu_P(v)\le 1/f(n),$ из чего следует, что $-\ln{\mu_P(v)}\ge \ln{f(n)}=g(n),$ а значит, 
    \begin{align*}
    g(n) \cdot \mu_P(\overline{A_f}(n))\le -\sum_{v\in \overline{A_f}(n)}\mu_P(v)\cdot \ln{\mu_P(v)}\le  H_n\le O(\ln^2{n}).
    \end{align*}
    Из данного неравенства очевидно следует, что $\mu_P(\overline{A_f}(n))\to 0,$ что и требовалось.
\end{proof}

\section{Покрывающие последовательности вида $\exp{(c\cdot\ln^2{n})}$}

В данном разделе нам понадобятся следующие обозначения:

Пусть $n\in\mathbb{N}_0.$ Обозначим через $S_n$ множество перестановок множества $\overline{1,n}.$ Хорошо известно, что $|S_n|=n!.$ Также при $i\in\mathbb{Z}$ обозначим через $S_{n,i}$ множество перестановок множества $\overline{1,n}$ ровно с $i$ нечётными циклами, а через
$\mathbb{YF}_{n,i}$ --- множество вершин графа Юнга --- Фибоначчи ранга $n$ ровно с $i$ единицами. Очевидно, что при $i\notin \overline{n}$ выполняется $S_{n,i}=\mathbb{YF}_{n,i}=\varnothing.$

Кроме того, пусть $W(z)$ --- это $W$-функция Ламберта. В данной работе из её свойств используется только то, что ${1}/(2W(1/e))\approx 1,79556$ является решением одного конкретного уравнения.

Следующее предложение следует, например, из \cite[раздел~4.4]{Fo2}, но мы приведём элементарное доказательство.
\begin{Prop}\label{p2}
    Пусть $n\in\mathbb{N}_0,$ $i\in\overline{n}.$ Тогда
    \begin{gather*} \sum_{v\in\mathbb{YF}_{n,i}}D^2(v)=|S_{n,i}|,\qquad \mu_P(\mathbb{YF}_{n,i})=\frac{|S_{n,i}|}{n!}.
    \end{gather*}
\end{Prop}
\begin{proof}
Ясно, что эти два равенства равносильны. Будем доказывать первое из них индукцией по $n.$ База $n\in\{0,1\}$ очевидна. Докажем переход к $n\ge 2.$ 

Зафиксируем $i\in\overline{n},$ $a\in\overline{1,n}$ и рассмотрим такие перестановки $s\in S_{n,i},$ что $s(n)=a.$ Очевидно, что если $a=n,$ то таких перестановок ровно $|S_{n-1,i-1}|.$ Если же $a\in\overline{1,n-1},$ то при выкидывании пары элементов $n,a$ из цикла, которому они принадлежат, мы получим все перестановки множества $\overline{1,n-1}\setminus a,$ содержащие ровно $i$ нечётных циклов, причём каждую ровно $n-1$ раз, так как вернуть пару элементов $n,a$ мы можем ровно в $n-1$ место. Таким образом, для всех $i\in\overline{n}$ выполняется $|S_{n,i}|=|S_{n-1,i-1}|+(n-1)^2|S_{n-2,i}|.$
    
    Осталось заметить, что согласно \eqref{fd} выполняется
    \begin{align*}
    \sum_{v\in\mathbb{YF}_{n,i}}D^2(v)=&\sum_{v\in\mathbb{YF}_{n-1,i-1}}D^2(1v)+\sum_{v\in\mathbb{YF}_{n-2,i}}D^2(2v)=\\
    =&\sum_{v\in\mathbb{YF}_{n-1,i-1}}D^2(v)+(n-1)^2\cdot\sum_{v\in\mathbb{YF}_{n-2,i}}D^2(v)
    \end{align*}
и применить предположение индукции.
\end{proof}

Экспоненциальную производящую функцию последовательности $|S_n|=n!$ числа перестановок $n$ элементов несложно вычислить и непосредственно, но мы напомним, как она связана с разбиением перестановок на циклы. 

Для каждого $k\in \mathbb{N}$ назовём $k$-циклом ориентированный граф
на $k$ вершинах, представляющий собой цикл (при $k=1$ этот граф ---
петля). 
Заметим, что при $k\in\mathbb{N}$ экспоненциальная производящая функция (по $x$) количества $k$-циклов на множестве
вершин $\overline{1,n}$ равна $(k-1)!\cdot x^k/k!=x^k/k$. 
Экспоненциальная производящая функция количества графов, являющихся
$k$-циклом при некотором $k$, равна, таким образом,
$\sum_k x^k/k$.

По базовой комбинаторной формуле ALL=exp(CONNECTED) экспоненциальная производящая функция количества ориентированных графов на множестве вершин 
$\overline{1,n}$, в которых каждая компонента связности является $k$-циклом при некотором $k$, равна $\exp{(\sum x^k/k)}.$ 
Но такие графы в точности соответствуют перестановкам
$\overline{1,n}$. Таким образом,

\begin{gather*}
    \sum_{n=0}^\infty|S_n|\cdot\frac{x^n}{n!}=\exp{\sum_{k=1}^\infty\frac{x^k}{k}}=\exp{(-\ln{(1-x))}}=\frac{1}{1-x}.
\end{gather*}

Теперь вычислим двойную экспоненциальную производящую функцию 
числа перестановок с учётом количества нечётных циклов.
    \begin{Prop}\label{p3}
        Для формальных степенных рядов имеет место соотношение
        \begin{gather}\label{p31}
        \sum_{n=0}^\infty\sum_{i=0}^n|S_{n,i}|\cdot t^i\cdot\frac{x^n}{n!}=\frac{\exp{\left(\frac{t}{2}\cdot\ln{\frac{1+x}{1-x}}\right)}}{\sqrt{1-x^2}}.
        \end{gather}    
        Более того, \eqref{p31} выполняется при всех вещественных 
        $x\in (0,1)$ и 
        $t>0$.
    \end{Prop}
    \begin{proof} Сначала докажем формальное равенство \eqref{p31}
    При чётных натуральных $k$ положим $r_k:=x^k/k,$ а при нечётных натуральных $k$ положим $r_k:=t \cdot x^k/k.$ Аналогично вышеизложенному рассуждению несложно заметить, что левая часть \eqref{p31} равняется
    \begin{align*}
        \exp{\sum_{k=1}^\infty r_k}=&\exp{\left(
        \sum_{k=1}^\infty \frac{x^k}{k}+(t-1)\cdot\sum_{k=1}^\infty \frac{x^{2k-1}}{2k-1}\right)}\\
        =&\exp{\left(-\ln{(1-x)}+\frac{t-1}{2}\cdot\ln{\frac{1+x}{1-x}}\right)},
    \end{align*}
что равносильно \eqref{p31}. Вывод функционального соотношения 
из формального проводится стандартным образом: 
при данном $t>0$ правая часть \eqref{p31} задаётся сходящимся
при $|x|<1$
степенным рядом по $x$. Степенной ряд слева также сходится,
поскольку его коэффициенты --- многочлены от $t$, причём 
в силу формального равенства \eqref{p31} --- те же,
что и справа. Значит, суммы этих рядов равны и конечны.
От порядка слагаемых сумма слева не зависит, поскольку они неотрицательны.
    \end{proof}
    Далее мы утверждаем, что при больших $n$ мера Планшереля на уровне $n$ сосредоточена на вершинах с достаточно малым числом единиц:
\begin{Lemma}\label{l1}
    Пусть $c>1/(2W(1/e))\approx 1,79556.$ Тогда
    \begin{gather*}
    \sum_{i=\lceil c\cdot\ln{n} \rceil}^n \mu_P(\mathbb{YF}_{n,i})\xrightarrow[n\to\infty]{}0.
    \end{gather*}                
\end{Lemma}
\begin{proof}
Из предложений \ref{p2} и \ref{p3} несложно заметить, что для любых $n,m\in\mathbb{N}_0,$ а также $x\in(0,1),$ $t\ge 1$ выполняется
    \begin{gather}\label{s2}
    \sum_{i=m}^n \mu_P(\mathbb{YF}_{n,i})\le \frac{\exp{\left(\frac{t}{2}\cdot\ln{\frac{1+x}{1-x}}\right)}}{\sqrt{1-x^2}\cdot t^m \cdot x^n}.
    \end{gather}            

При каждом натуральном $n$ подставим в \eqref{s2} значения 
    \begin{gather*}
        x=1-\frac{1}{n},\qquad t=\frac{2c\cdot \ln{n}}{\ln{\frac{1+x}{1-x}}}=\frac{2c\cdot\ln{n}}{\ln{(2n-1)}},\qquad m=\lceil c\cdot\ln{n} \rceil.
    \end{gather*}
Очевидно, что при достаточно больших $n$ условие $t\ge 1$ выполняется. Таким образом, существует такая константа $C',$ что при всех натуральных $n\ge 2$ выполняется
\begin{align*}
    \sum_{i=\lceil c\cdot\ln{n} \rceil}^n \mu_P(\mathbb{YF}_{n,i})\le& \frac{\exp{(c\cdot \ln{n})}}{\sqrt{\frac{2}{n}-\frac{1}{n^2}}\cdot\left(\frac{2c\cdot\ln{n}}{\ln{(2n-1)}}\right)^{\lceil c\cdot\ln{n} \rceil} \cdot \left(1-\frac{1}{n}\right)^n}\\
    \le& C'\cdot\sqrt{n}\cdot \left(\frac{e}{2c}\right)^{c\cdot \ln{n}}= C'\cdot\left(\frac{\exp{\left(1+\frac{1}{2c}\right)}}{2c}\right)^{c\cdot\ln{n}}.
\end{align*}
Очевидно, на отрезке $(0,\infty)$ функция $\exp(1+1/(2c))/(2c)$ монотонно убывает, стремясь к нулю. Равна единице она как раз в точке $c=1/(2W(1/e)).$ Лемма очевидно следует.
\end{proof}

Теперь при больших $n$ оценим сверху число вершин на уровне $n,$ про которые мы доказали, что мера Планшереля сосредоточена на них:
\begin{Prop}\label{p4}
Пусть $c>1/(2W(1/e))\approx 1,79556.$ Тогда при достаточно больших $n$
    \begin{gather*}
        \sum_{i=0}^{\lceil c\cdot\ln{n} \rceil-1}|\mathbb{YF}_{n,i}|\le n^{c\cdot\ln{n}}.
    \end{gather*}        
\end{Prop}
    \begin{proof}
        Пусть $n$ чётно, $i\in\overline{n}.$ Несложно заметить, что если $i$ нечётно, то множество $\mathbb{YF}_{n,i}$ пусто, а если $i$ чётное и $i=2j,$ то выполняется $|\mathbb{YF}_{n,i}|=\binom{n/2+j}{2j},$ так как в $\mathbb{YF}_{n,i}$ находятся все вершины с $n/2+j$ цифрами, из которых $2j$ единиц. Таким образом, при достаточно больших $n$
    \begin{gather*}
        \sum_{i=0}^{\lceil c\cdot\ln{n} \rceil-1}|\mathbb{YF}_{n,i}|\le  \sum_{j=0}^{\left\lfloor\frac{ c\cdot\ln{n} }{2}\right\rfloor}\binom{\frac{n}{2}+j}{2j}\le c\cdot \ln{n}\cdot\left(\frac{n}{2}+\frac{c\cdot\ln{n}}{2}\right)^{c\cdot\ln{n}} \le n^{c\cdot\ln{n}}.
    \end{gather*}        
    Случай нечётного $n$ вполне аналогичен.
        \end{proof}

Из леммы \ref{l1} и предложения \ref{p4} сразу следует
\begin{theorem}\label{t2}
Пусть $c>1/(2W(1/e))\approx 1,79556.$ Тогда последовательность 
$\{f(n)\}_{n=1}^\infty={\exp {(c\cdot\ln^2{n})}},$ является покрывающей.
\end{theorem}

\section{Не покрывающие последовательности вида $\exp{(c\cdot\ln^2{n})}$}

В данном разделе нам понадобится следующее обозначение: 

Пусть $a\in(1,2),$ $n\in\mathbb{N}_0.$ Тогда 
    \begin{gather*}
        M_a(n):=\sum_{v\in\mathbb{YF}_n}\mu^a_P(v)=\sum_{v\in\mathbb{YF}_n}\frac{D^{2a}(v)}{(n!)^a}.
    \end{gather*}
\begin{Prop}\label{p5}
    Пусть $a\in(1,2).$ Тогда $M_a(0)=M_a(1)=1,$ при $n\ge 2$ выполняется
    \begin{gather*}
        M_a(n)=\frac{1}{n^a}\cdot M_a(n-1)+\frac{(n-1)^a}{n^a}\cdot M_a(n-2).
    \end{gather*}
\end{Prop}
\begin{proof}
Равенство $M_a(0)=M_a(1)=1$ понятно. При $n\ge 2$ воспользуемся \eqref{fd}:
\begin{align*}
M_a(n)=&\sum_{v\in\mathbb{YF}_{n-1}}\frac{D^{2a}(1v)}{(n!)^a}+\sum_{v\in\mathbb{YF}_{n-2}}\frac{D^{2a}(2v)}{(n!)^a}\\
=&\frac{1}{n^a}\cdot\sum_{v\in\mathbb{YF}_{n-1}}\frac{D^{2a}(v)}{((n-1)!)^a}+\frac{(n-1)^a}{n^a}\cdot\sum_{v\in\mathbb{YF}_{n-2}}\frac{D^{2a}(v)}{((n-2)!)^a}.
\end{align*}
Предложение следует.
\end{proof}

\begin{Lemma}\label{l2}
    Пусть $a\in(1,2),$ $n\in\mathbb{N}_0.$ Тогда 
    \begin{gather*} 
    M_a(n)\le\frac{\exp{\frac{a}{a-1}}}{n^{a/2}}.
    \end{gather*}    
\end{Lemma}
\begin{proof}
    Для начала докажем, что при всех $a\in(1,2),$ $n\in\mathbb{N}_0$ выполняется 
    \begin{gather}\label{p6} 
    M_a(n)\le\frac{\prod_{m=1}^n(1+{1}/{m^a})}{(n+1)^{a/2}}.
    \end{gather}    
    При фиксированном $a\in(1,2)$ докажем \eqref{p6} индукцией по $n.$ База $n\in\{0,1\}$ очевидна. Докажем переход к $n\ge 2.$

    Применим предложение \ref{p5}, предположение индукции и обобщённое неравенство Бернулли:
    \begin{align*}      
    M_a(n)=&\frac{1}{n^a}\cdot M_a(n-1)+\frac{(n-1)^a}{n^a}\cdot M_a(n-2)\\
\le&\frac{1}{n^a}\cdot\frac{\prod_{m=1}^{n-1}(1+1/m^a)}{n^{a/2}}+\frac{(n-1)^a}{n^a}\cdot\frac{\prod_{m=1}^{n-2}(1+1/m^a)}{(n-1)^{a/2}}\\
\le& \frac{\prod_{m=1}^{n-1}(1+1/m^a)}{(n+1)^{a/2}}\cdot\left( \frac{1}{n^a}\cdot\frac{(n+1)^{a/2}}{n^{a/2}}+\frac{(n^2-1)^{a/2}}{n^a}\right)\\
\le& \frac{\prod_{m=1}^{n-1}(1+1/m^a)}{(n+1)^{a/2}}\cdot\left(\frac{1}{n^a}\cdot\left(1+\frac{a}{2n}\right)+\left(1-\frac{a}{2n^2}\right)\right).
\end{align*}    
Переход следует.

Теперь применим \eqref{p6} для доказательства леммы:
\begin{gather*}
    M_a(n)\le \frac{\prod_{m=1}^n(1+1/m^a)}{(n+1)^{a/2}} 
    \le\frac{\exp{{\sum_{m=1}^\infty {1/m^a}}}}{n^{a/2}}\le \frac{\exp{\left(1+\int_{1}^\infty{dm}/{m^a}\right)}}{n^{a/2}}=\frac{\exp{\frac{a}{a-1}}}{n^{a/2}}.
\end{gather*}    
\end{proof}

\begin{theorem}\label{t3}
Последовательность $\{f(n)\}_{n=1}^\infty={\exp {(1/16\cdot\ln^2{n})}},$ не является покрывающей.
\end{theorem}
\begin{proof}
Заметим, что при любых $a\in(1,2),$ $n\in\mathbb{N}_0,$ в силу леммы \ref{l2} и неравенства Йенсена для выпуклых вниз функций, выполняется 
    \begin{gather*}
        \frac{\exp{\frac{a}{a-1}}}{n^{a/2}}\ge M_a(n)\ge\sum_{v\in A_f(n)}\mu_P^a(v)\ge\left(\frac{\mu_P(A_f(n))}{|A_f(n)|}\right)^a\cdot {|A_f(n)|}\ge \frac{\mu_P^a(A_f(n))}{\exp{(\frac{1}{16}\ln^2{n}\cdot(a-1) )}},
    \end{gather*}
а значит, выполняется
\begin{gather}\label{final}
    \mu_P(A_f(n))\le \frac{\exp{\frac{1}{a-1}}}{\exp{\left(\frac{1}{2}\ln{n}\right)}}\cdot \exp{\left(\frac{1}{16}\ln^2{n}\cdot\frac{a-1}{a}\right)}.
\end{gather}

При каждом натуральном $n\ge 55$ подставим в \eqref{final} значения $a=1+4/\ln{n}:$
\begin{gather*}
    \mu_P(A_f(n))\le \frac{\exp{(\frac{1}{4}\ln{n})}}{\exp{\left(\frac{1}{2}\ln{n}\right)}}\cdot \exp{\frac{\frac{1}{4}\ln^2{n}}{\ln{n}+4}}=\exp{\left(-\frac{\ln{n}}{\ln{n}+4}\right)}.
\end{gather*}
Из данного неравенства ясно, что при достаточно больших $n$ значение $\mu_P(A_f(n))$ меньше, например, чем $1/2.$ Теорема очевидно следует.
\end{proof}

\section{Благодарности}
Я очень признателен моему научному руководителю Фёдору Владимировичу Петрову за постановку задачи, проявленное терпение при проверке доказательства и редактировании статьи, а также за моральную поддержку на протяжении всего периода работы. 

Работа выполнена при поддержке СПбГУ, шифр проекта 159547043.

\selectlanguage{english}

\title{On how many Fibonacci words is the Plancherel measure concentrated}

\author{V.~Yu.~Evtushevsky}

\maketitle

\setcounter{page}{1}

\etocsettagdepth{en}{all}
\etocsettagdepth{ru}{none}

\tableofcontents

\theoremstyle{style2}
\newtheorem{enLemma}{Lemma}

\theoremstyle{style2}
\newtheorem{enAlg}{Algorithm}

\theoremstyle{style2}
\newtheorem{enCol}{Corollary}

\theoremstyle{style2}
\newtheorem{entheorem}{Theorem}

\theoremstyle{style2}
\newtheorem{enDef}{Definition}

\theoremstyle{style2}
\newtheorem{enProp}{Proposition}

\theoremstyle{style2}
\newtheorem{enProblem}{Problem}

\theoremstyle{style2}
\newtheorem{enZam}{Remark}

\theoremstyle{style2}
\newtheorem{enOboz}{Notation}

\theoremstyle{style2}
\newtheorem{enEx}{Example}

\theoremstyle{style2}
\newtheorem{enNab}{Observation}

\setcounter{section}{0}
\setcounter{figure}{0} 
\setcounter{equation}{0}

\begin{abstract}
On the set $\mathbb{YF}_n$ of words in $1$ and $2$ 
with digit sum $n$, we consider the probability distribution
$\mu_P$,
corresponding to the Plancherel measure on paths in the
Young--Fibonacci graph. We study the asymptotics as $n$ grows of the entropy
 of the distribution $\mu_P$ and of the minimum number of
words on which $\mu_P$ is almost entirely concentrated. 
\end{abstract}

\etocdepthtag.toc{en}
\section{Introduction}
Consider words over the alphabet $\{1,2\}$ with a given digit sum $n.$ As is well known, their number is the Fibonacci number $F_{n+1}$ ($F_0=0,$ $F_1=1,$ $F_{k+2}=F_{k+1}+F_k$), and this is the most common combinatorial interpretation of Fibonacci numbers. One can also think of tilings of a $2\times n$ strip by dominoes of sizes $1\times 2$ and $2\times 1,$ associating twos with pairs of horizontal dominoes and ones with vertical dominoes.

Introduce a partial order on this set of words: we say that a word $v$ \emph{precedes} a word $w$ if, after removing the longest common suffix, the word $w$ has at least as many twos remaining as the word $v$ has digits remaining. 

This is indeed a partial order; moreover, the corresponding partially ordered set is a modular lattice, known as the Young--Fibonacci lattice.

The Young--Fibonacci graph (shown in Fig. \ref{enyf}) is the Hasse diagram of this lattice. It is a graded graph, which we represent as growing from bottom to top, starting with the empty word. The grading is given by the digit sum function. Let us describe explicitly how the directed edges are defined. The ``upward'' edges from a given word $v$ lead to words obtained from $v$ by one of two operations:
\begin{enumerate}
    \item replace the leftmost one by a two;
    \item insert a one at a position
    with no ones to its left.
\end{enumerate}

\begin{figure}[htbp]
    \centering
    \includegraphics[width=12cm, height=10cm]{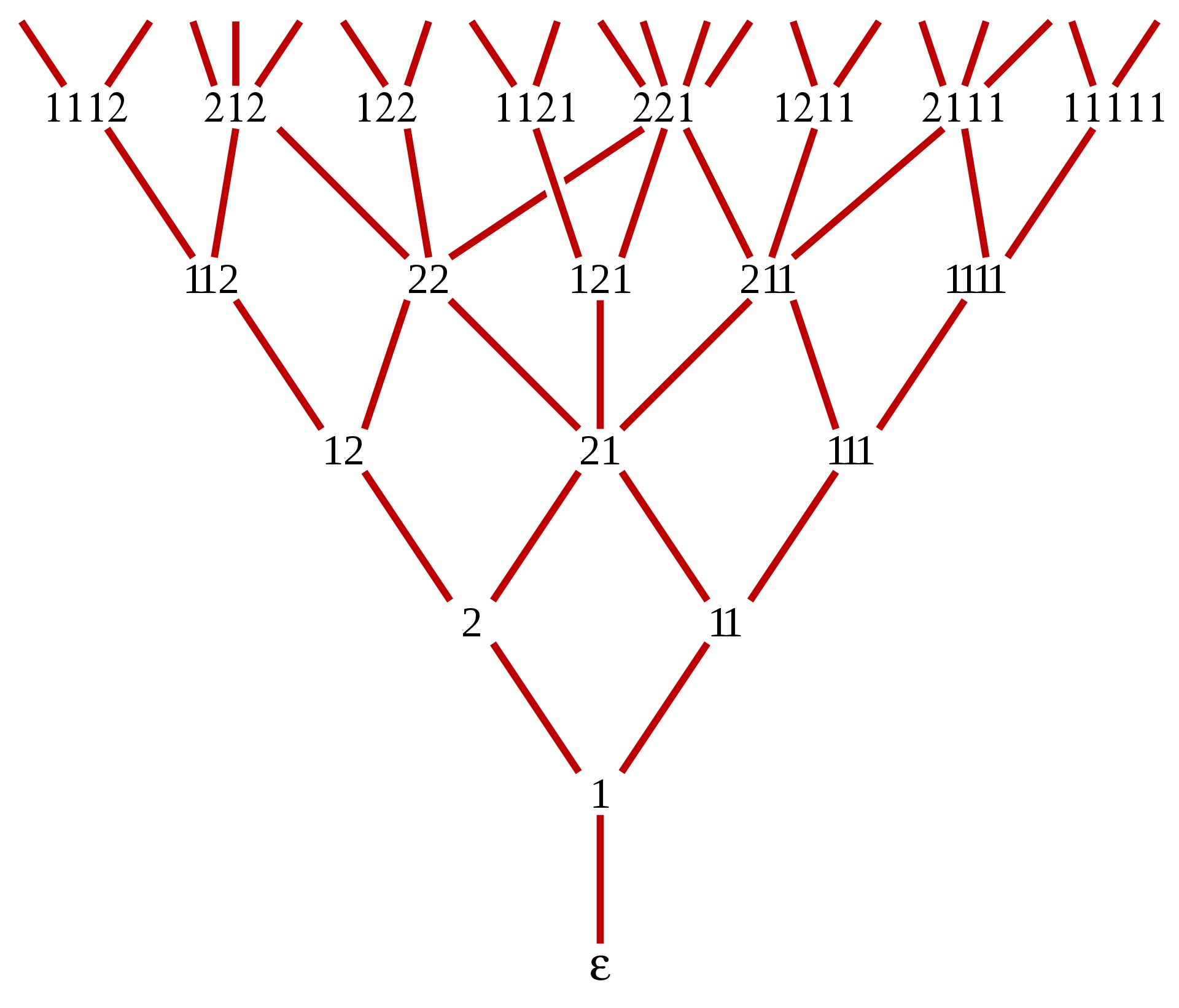}
    \caption{The Young--Fibonacci graph}
    \label{enyf}
\end{figure}

In this paper, we regard the Young--Fibonacci graph as directed and containing only ``upward'' edges.

In addition to being modular, this graph is $1$-differential, that is, for each vertex, the out-degree exceeds the in-degree by $1$.

The study of the graded Young--Fibonacci graph was initiated in 1988 simultaneously and independently by the mathematicians Richard Stanley \cite{enSt} and Sergei Vladimirovich Fomin \cite{enFo}.

One reason for the interest in it is that there are only two rooted locally finite $1$-differential modular lattices; the other is the lattice of Young diagrams, which plays a key role in the representation theory of the symmetric group \cite[Theorem~1.3]{enBur}. 

The key questions about graded graphs concern central measures on the space of (infinite) paths in the graph. This point of view was systematically developed in the works of Anatoly Moiseevich Vershik; we refer the reader to his recent survey \cite{enVer} and the references therein.

Among central measures, one singles out those that are limits of measures induced by paths to distant vertices---the so-called \emph{Martin boundary} of the graph.

The boundary of the path space of the Young--Fibonacci graph was studied in the paper by Frederick Goodman and Sergei Vasilievich Kerov (2000) \cite{enGK}. They used Okada's algebraic formalism \cite{enOkada}.

Goodman and Kerov proved that the list of central measures of interest to us is exhausted by the following measures:
\begin{enumerate}
    \item The Plancherel measure: for a given vertex $v$ of rank $n,$ the measure of the set of paths passing through this vertex is ${D(v)^2}/{n!}$, where $D(v)$ is the number of ``upward'' paths from the root to the vertex $v$.
    \item The measures $\mu_{w,\beta},$ parametrized by an infinite word $w,$ containing ``sufficiently few'' twos, and a number $\beta\in(0,1]$.
\end{enumerate}

This paper is devoted to the Plancherel measure. This measure was studied, in particular, by S.~V.~Kerov and Alexander Vasilievich Gnedin \cite{enKerGned}, who proved
its ergodicity. In \cite{enBE}, Ivan Alekseevich Bochkov
and the author obtained explicit formulas for the remaining measures and proved their
ergodicity.

For each $n$,
the Plancherel measure defines a probability measure on the set $\mathbb{YF}_n$ 
of vertices of the Young--Fibonacci graph of rank $n.$ This paper
studies precisely this sequence of probability measures.

In Section 3, we compute the asymptotics (up to $O(1)$) of the entropies of these measures. In Section 4, we prove that there exists a constant $c'\approx 1,79556,$ such that, for all $c>c'$ and large $n$, the Plancherel measure is \emph{concentrated} on $\lfloor\exp{(c\cdot\ln^2{n})}\rfloor$ vertices of the Young--Fibonacci graph of rank $n.$ In Section 5, we prove that, for large $n$, the Plancherel measure
\emph{is not concentrated} on any $\lfloor\exp{({1}/{16}\cdot\ln^2{n})}\rfloor$ vertices of the Young--Fibonacci graph of rank $n.$

Note that the entropy of the Plancherel measure of the Young graph
(defined on the set $\mathbb{Y}_n$ of Young diagrams with $n$
cells) grows considerably
faster---as $\Theta(\sqrt{n})$ \cite{enVK}, so it is not concentrated
on fewer than $\exp({\rm const}\sqrt{n})$ vertices. This is all the
more surprising because the \emph{total} number of vertices of rank $n$
in the Young--Fibonacci graph (that is, the Fibonacci number $F_{n+1}$) grows, on the contrary, considerably faster than the total number of vertices of rank $n$ in the Young graph (that is, the number of partitions
$p(n)$ of the number $n$).

    \section{Basic concepts and notation}

Let us introduce the basic concepts and notation.

The sets of nonnegative integers and positive integers (natural numbers) are denoted by $\mathbb{N}_0$ and $\mathbb{N}$, respectively. 
For $m,n\in \mathbb{N}_0:$ $m\leqslant n$, put
\begin{gather*}
	\overline{m,n}:=\{m,m+1,\ldots,n\},\quad \overline{n}:=\overline{0,n}=\{0,1,\ldots,n\}.
\end{gather*}

Let $\mathbb{YF}$ be the Young--Fibonacci graph; its vertices are finite words in $1$ and $2.$ In what follows, $\mathbb{YF}$, as a set, is precisely the set of vertices of the Young--Fibonacci graph. For $v,w\in\mathbb{YF}$, the word $vw$ is defined as the concatenation of $v$ and $w.$ 

For $v \in \mathbb{YF}$, denote the digit sum of the word $v$ by $|v|$. The number $|v|$ is the \emph{rank} of the vertex $v$ in the graph $\mathbb{YF}$ (we assume that the empty word has rank $0$ and that moving upward along an edge increases the rank by $1$). 

For a nonnegative integer $n$, denote the set of vertices of the graph $\mathbb{YF}$ of rank $n$ by
\begin{gather*}
	\mathbb{YF}_n:=\{v\in\mathbb{YF}:\;|v|=n\}.
\end{gather*}	
In this paper, we regard the graph $\mathbb{YF}$ as directed, with edges from vertices of smaller rank to vertices of larger rank.

For $v\in \mathbb{YF}$, denote the number of paths from $\varepsilon$ to $v$ by $D(v).$ It is well known that
\begin{gather} \label{enfd}
    D(\varepsilon)=1,\qquad D(1v)= D(v),\qquad D(2v)=(|v|+1)\cdot D(v).   
\end{gather}
For a proof, see, for example, the paper by Gnedin and Kerov \cite[formula~3.4]{enKerGned}. 

In particular, it follows directly from the recursive formula for $D(v)$ that, for all $n\in\mathbb{N}_0$,
\begin{gather}\label{enfactorial}
    \sum_{v\in\mathbb{YF}_n} D^2(v)=n!.
\end{gather}

The \emph{Plancherel measure} $\mu_P$ of the Young--Fibonacci graph is a probability measure on the space of infinite paths from the root such that, for any finite path $t=(\varepsilon,\ldots,v),$ where $v\in\mathbb{YF}_n,$ the measure of the set of all infinite paths beginning with $t,$ is $D(v)/n!.$ Denote by $\mu_P(v)=D^2(v)/n!$ the measure of all paths passing through $v.$

Thus, for a given $n\in \mathbb{N}_0$, the function $\mu_P(v)$ defines a probability measure on the set $\mathbb{YF}_n.$ Expressions such as $\mu_P(A)$ for subsets $A\subseteq\mathbb{YF}_n$ are understood in this sense.

Denote the \emph{entropy} of this probability measure on the set $\mathbb{YF}_n$ by $H_n.$ That is,
\begin{gather*}
    H_n:=-\sum_{v\in\mathbb{YF}_n}\mu_P(v)\cdot \ln{\mu_P(v)}= \sum_{v\in\mathbb{YF}_n}\frac{D^2(v)}{n!}\cdot(\ln{n!}-2\ln{D(v)}).
\end{gather*}

Also, let $\{f(n)\}_{n=1}^\infty$ be a sequence of nonnegative numbers. Then, for a fixed $n$, denote by $A_{f}(n)$ a set of $\lfloor f(n)\rfloor$ vertices of the Young--Fibonacci graph of rank $n$ with the largest values of the Plancherel measure (in general, the set $A_f(n)$ is not uniquely defined; in the case of ambiguity, we fix an arbitrary choice). If $\lfloor f(n)\rfloor > |\mathbb{YF}_n|,$ we assume that $A_{f}(n)=\mathbb{YF}_n.$ In addition, let $\overline{A_{f}}(n)=\mathbb{YF}_n\setminus A_{f}(n).$ We also call the sequence $f$ \emph{covering} if $\lim_{n\to\infty}\mu_{P}(A_f(n))=1$ (or, equivalently, $\lim_{n\to\infty}\mu_{P}(\overline{A_f}(n))=0)$ and \emph{non-covering} if $\varlimsup_{n\to\infty}\mu_{P}(A_f(n))<1.$

\section{Entropy of the sequence of measures}

We first prove a recursive formula for the entropies of the probability measures defined by the Plancherel measure on the sets $\mathbb{YF}_n:$ 
\begin{enProp}\label{enentrec}
$H_0=H_1=0,$ and, for $n\ge 2$, we have
\begin{gather*}\label{enentrec1}
H_n=\frac{1}{n}\cdot H_{n-1}+\frac{n-1}{n}\cdot H_{n-2}+\ln{n}-\frac{n-1}{n}\cdot \ln{(n-1)}.    
\end{gather*}    
\end{enProp}
\begin{proof}
The equality $H_0=H_1=0$ is clear. For $n\ge 2$, we use \eqref{enfd}:
\begin{align*}
H_n=&\sum_{v\in\mathbb{YF}_{n-1}}\frac{D^2(1v)}{n!}\cdot(\ln{n!}-2\ln{D(1v)})+\sum_{v\in\mathbb{YF}_{n-2}}\frac{D^2(2v)}{n!}\cdot(\ln{n!}-2\ln{D(2v)})\\
=&\frac{1}{n}\cdot\sum_{v\in\mathbb{YF}_{n-1}}\frac{D^2(v)}{(n-1)!}\cdot(\ln{(n-1)!}-2\ln{D(v)}+\ln{n})\\
&+\frac{n-1}{n}\cdot\sum_{v\in\mathbb{YF}_{n-2}}\frac{D^2(v)}{(n-2)!}\cdot(\ln{(n-2)!}-2\ln{D(v)}+\ln{n}-\ln{(n-1)}).
\end{align*}
The proposition now follows from \eqref{enfactorial}.
\end{proof}

It follows directly from the Mercator series expansion of the logarithm $\ln{(1-c/n)}$ that, for $n\ge 3$, we have
\begin{align}
    \ln{(n-1)}=\ln{n}-\frac{1}{n}+O\left(\frac{1}{n^2}\right),&\qquad \ln{(n-2)}=\ln{n}-\frac{2}{n}+O\left(\frac{1}{n^2}\right);\notag\\
    \ln^2{(n-1)}=\ln^2{n}-\frac{2\ln{n}}{n}+O\left(\frac{\ln{n}}{n^2}\right),&\qquad \ln^2{(n-2)}=\ln^2{n}-\frac{4\ln{n}}{n}+O\left(\frac{\ln{n}}{n^2}\right).\label{enf1}    
\end{align} 

We are ready to compute, up to $O(1)$, the asymptotics of the entropies of the probability measures defined by the Plancherel measure on the sets $\mathbb{YF}_n:$ 
\begin{entheorem}\label{ent1}
The following asymptotic relation holds:
    \begin{gather*}\label{ent11}
    H_n=\frac{1}{4}\ln^2{n}+\frac{1}{2}\ln{n}+O(1).
    \end{gather*}
\end{entheorem}
\begin{proof}
For each natural number $n$, put $h_n:=H_n-1/4\ln^2{n}-1/2\ln{n}.$ Proposition \ref{enentrec} and formulas \eqref{enf1} imply the following asymptotic relation:
\begin{align*}
    |h_n|=&\biggl|\frac{1}{n}\cdot H_{n-1}+\frac{n-1}{n}\cdot H_{n-2}+\ln{n}-\frac{n-1}{n}\cdot\ln{(n-1)}-\left(\frac{1}{4}\ln^2{n}+\frac{1}{2}\ln{n}\right)\biggr|\\
    =&\biggl|\frac{1}{n}\cdot\left(\frac{1}{4}\ln^2{(n-1)}+\frac{1}{2}\ln{(n-1)}+h_{n-1}\right)+\frac{n-1}{n}\cdot\left(\frac{1}{4}\ln^2{(n-2)}+\frac{1}{2}\ln{(n-2)}+h_{n-2}\right)\\
    &-\frac{1}{4}\ln^2{n}+\frac{1}{2}\ln{n}-\frac{n-1}{n}\cdot\ln{(n-1)}\biggr|\\
    \le& O\left(\frac{\ln{n}}{n^2}\right)+\biggl|\frac{1}{n}\cdot h_{n-1}+\frac{n-1}{n}\cdot h_{n-2}\biggr|+\biggl|-\frac{1}{4}\ln^2{n}+\frac{1}{4n}\cdot\ln^2{n}\\
    &+\frac{n-1}{4n}\cdot\left(\ln^2{n}-\frac{4\ln{n}}{n}\right)+\frac{1}{2}\ln{n}-\frac{2n-3}{2n}\cdot\left(\ln{n}-\frac{1}{n}\right)+\frac{n-1}{2n}\cdot\left(\ln{n}-\frac{2}{n}\right)\biggr|\\
    \le& O\left(\frac{\ln{n}}{n^2}\right)+\max{(|h_{n-1}|,|h_{n-2}|)}.
\end{align*}

The theorem now follows from the convergence of the series $\sum_{n=1}^\infty \ln{n}/n^2.$
\end{proof}

\begin{enCol}\label{ens1} 
For any sequence $\{g(n)\}_{n=1}^\infty$ such that $g(n)/\ln^2{n}\to \infty,$ the sequence $f(n)=\exp{g(n)}$ is covering.
\end{enCol}
\begin{proof}
Fix a natural number $n.$ It is easy to see that, for each vertex $v\in\overline{A_f}(n)$, we have $\mu_P(v)\le 1/f(n),$ which implies that $-\ln{\mu_P(v)}\ge \ln{f(n)}=g(n),$ and hence,
    \begin{align*}
    g(n) \cdot \mu_P(\overline{A_f}(n))\le -\sum_{v\in \overline{A_f}(n)}\mu_P(v)\cdot \ln{\mu_P(v)}\le  H_n\le O(\ln^2{n}).
    \end{align*}
    It follows immediately from this inequality that $\mu_P(\overline{A_f}(n))\to 0,$ as required.
\end{proof}

\section{Covering sequences of the form $\exp{(c\cdot\ln^2{n})}$}

In this section, we need the following notation:

Let $n\in\mathbb{N}_0.$ Denote by $S_n$ the set of permutations of the set $\overline{1,n}.$ It is well known that $|S_n|=n!.$ Also, for $i\in\mathbb{Z}$, denote by $S_{n,i}$ the set of permutations of the set $\overline{1,n}$ with exactly $i$ odd cycles, and by
$\mathbb{YF}_{n,i}$ the set of vertices of the Young--Fibonacci graph of rank $n$ with exactly $i$ ones. Clearly, for $i\notin \overline{n}$, we have $S_{n,i}=\mathbb{YF}_{n,i}=\varnothing.$

In addition, let $W(z)$ be the Lambert $W$ function. The only property of this function used in this paper is that ${1}/(2W(1/e))\approx 1,79556$ is a solution of one particular equation.

The following proposition can be deduced, for example, from \cite[Section~4.4]{enFo2}, but we give an elementary proof.
\begin{enProp}\label{enp2}
    Let $n\in\mathbb{N}_0,$ $i\in\overline{n}.$ Then
    \begin{gather*} \sum_{v\in\mathbb{YF}_{n,i}}D^2(v)=|S_{n,i}|,\qquad \mu_P(\mathbb{YF}_{n,i})=\frac{|S_{n,i}|}{n!}.
    \end{gather*}
\end{enProp}
\begin{proof}
Clearly, these two equalities are equivalent. We prove the first by induction on $n.$ The base case $n\in\{0,1\}$ is obvious. We prove the induction step for $n\ge 2.$ 

Fix $i\in\overline{n},$ $a\in\overline{1,n}$ and consider permutations $s\in S_{n,i},$ such that $s(n)=a.$ Clearly, if $a=n,$ the number of such permutations is exactly $|S_{n-1,i-1}|.$ If $a\in\overline{1,n-1},$ then, by removing the pair of elements $n,a$ from the cycle to which they belong, we obtain all permutations of the set $\overline{1,n-1}\setminus a,$ containing exactly $i$ odd cycles, each exactly $n-1$ times, since we can reinsert the pair of elements $n,a$ in exactly $n-1$ positions. Thus, for all $i\in\overline{n}$, we have $|S_{n,i}|=|S_{n-1,i-1}|+(n-1)^2|S_{n-2,i}|.$
    
    It remains to note that, by \eqref{enfd}, we have
    \begin{align*}
    \sum_{v\in\mathbb{YF}_{n,i}}D^2(v)=&\sum_{v\in\mathbb{YF}_{n-1,i-1}}D^2(1v)+\sum_{v\in\mathbb{YF}_{n-2,i}}D^2(2v)=\\
    =&\sum_{v\in\mathbb{YF}_{n-1,i-1}}D^2(v)+(n-1)^2\cdot\sum_{v\in\mathbb{YF}_{n-2,i}}D^2(v)
    \end{align*}
and to apply the induction hypothesis.
\end{proof}

The exponential generating function of the sequence $|S_n|=n!$ counting permutations of $n$ elements is easy to compute directly, but we recall how it is related to the decomposition of permutations into cycles. 

For each $k\in \mathbb{N}$, by a $k$-cycle we mean a directed graph
on $k$ vertices that forms a cycle (for $k=1$, this graph is
a loop). 
Note that, for $k\in\mathbb{N}$, the exponential generating function (in $x$) of the number of $k$-cycles on the
vertex set $\overline{1,n}$ is $(k-1)!\cdot x^k/k!=x^k/k$. 
Thus, the exponential generating function of the number of graphs that are
a $k$-cycle for some $k$ is
$\sum_k x^k/k$.

By the basic combinatorial formula ALL=exp(CONNECTED), the exponential generating function of the number of directed graphs on the vertex set
$\overline{1,n}$, in which every connected component is a $k$-cycle for some $k$, is $\exp{(\sum x^k/k)}.$ 
But such graphs correspond exactly to permutations of
$\overline{1,n}$. Thus,

\begin{gather*}
    \sum_{n=0}^\infty|S_n|\cdot\frac{x^n}{n!}=\exp{\sum_{k=1}^\infty\frac{x^k}{k}}=\exp{(-\ln{(1-x))}}=\frac{1}{1-x}.
\end{gather*}

Now we compute the bivariate exponential generating function
of the number of permutations, taking into account the number of odd cycles.
    \begin{enProp}\label{enp3}
        The following identity holds for formal power series:
        \begin{gather}\label{enp31}
        \sum_{n=0}^\infty\sum_{i=0}^n|S_{n,i}|\cdot t^i\cdot\frac{x^n}{n!}=\frac{\exp{\left(\frac{t}{2}\cdot\ln{\frac{1+x}{1-x}}\right)}}{\sqrt{1-x^2}}.
        \end{gather}    
        Moreover, \eqref{enp31} holds for all real
        $x\in (0,1)$ and
        $t>0$.
    \end{enProp}
    \begin{proof} First, we prove the formal equality \eqref{enp31}.
    For even natural numbers $k$, put $r_k:=x^k/k,$ and for odd natural numbers $k$, put $r_k:=t \cdot x^k/k.$ By an argument similar to the one above, it is easy to see that the left-hand side of \eqref{enp31} equals
    \begin{align*}
        \exp{\sum_{k=1}^\infty r_k}=&\exp{\left(
        \sum_{k=1}^\infty \frac{x^k}{k}+(t-1)\cdot\sum_{k=1}^\infty \frac{x^{2k-1}}{2k-1}\right)}\\
        =&\exp{\left(-\ln{(1-x)}+\frac{t-1}{2}\cdot\ln{\frac{1+x}{1-x}}\right)},
    \end{align*}
which is equivalent to \eqref{enp31}. The functional identity is deduced
from the formal one in the standard way:
for a given $t>0$, the right-hand side of \eqref{enp31} is represented by a power series
convergent for $|x|<1$
in $x$. The power series on the left also converges,
because its coefficients are polynomials in $t$ and,
by the formal equality \eqref{enp31}, are the same
as those on the right. Thus, the sums of these series are equal and finite.
The sum on the left does not depend on the order of the terms, since they are nonnegative.
    \end{proof}
    Next, we assert that, for large $n$, the Plancherel measure on level $n$ is concentrated on vertices with a sufficiently small number of ones:
\begin{enLemma}\label{enl1}
    Let $c>1/(2W(1/e))\approx 1,79556.$ Then
    \begin{gather*}
    \sum_{i=\lceil c\cdot\ln{n} \rceil}^n \mu_P(\mathbb{YF}_{n,i})\xrightarrow[n\to\infty]{}0.
    \end{gather*}                
\end{enLemma}
\begin{proof}
From Propositions \ref{enp2} and \ref{enp3}, it is easy to see that, for any $n,m\in\mathbb{N}_0,$ and $x\in(0,1),$ $t\ge 1$, we have
    \begin{gather}\label{ens2}
    \sum_{i=m}^n \mu_P(\mathbb{YF}_{n,i})\le \frac{\exp{\left(\frac{t}{2}\cdot\ln{\frac{1+x}{1-x}}\right)}}{\sqrt{1-x^2}\cdot t^m \cdot x^n}.
    \end{gather}            

For each natural number $n$, substitute into \eqref{ens2} the values
    \begin{gather*}
        x=1-\frac{1}{n},\qquad t=\frac{2c\cdot \ln{n}}{\ln{\frac{1+x}{1-x}}}=\frac{2c\cdot\ln{n}}{\ln{(2n-1)}},\qquad m=\lceil c\cdot\ln{n} \rceil.
    \end{gather*}
Clearly, for sufficiently large $n$, the condition $t\ge 1$ holds. Thus, there exists a constant $C',$ such that, for all natural numbers $n\ge 2$, we have
\begin{align*}
    \sum_{i=\lceil c\cdot\ln{n} \rceil}^n \mu_P(\mathbb{YF}_{n,i})\le& \frac{\exp{(c\cdot \ln{n})}}{\sqrt{\frac{2}{n}-\frac{1}{n^2}}\cdot\left(\frac{2c\cdot\ln{n}}{\ln{(2n-1)}}\right)^{\lceil c\cdot\ln{n} \rceil} \cdot \left(1-\frac{1}{n}\right)^n}\\
    \le& C'\cdot\sqrt{n}\cdot \left(\frac{e}{2c}\right)^{c\cdot \ln{n}}= C'\cdot\left(\frac{\exp{\left(1+\frac{1}{2c}\right)}}{2c}\right)^{c\cdot\ln{n}}.
\end{align*}
Clearly, on the interval $(0,\infty)$, the function $\exp(1+1/(2c))/(2c)$ decreases monotonically to zero. It equals one precisely at $c=1/(2W(1/e)).$ The lemma follows immediately.
\end{proof}

Now, for large $n$, we bound from above the number of vertices on level $n,$ on which we have proved that the Plancherel measure is concentrated:
\begin{enProp}\label{enp4}
Let $c>1/(2W(1/e))\approx 1,79556.$ Then, for sufficiently large $n$,
    \begin{gather*}
        \sum_{i=0}^{\lceil c\cdot\ln{n} \rceil-1}|\mathbb{YF}_{n,i}|\le n^{c\cdot\ln{n}}.
    \end{gather*}        
\end{enProp}
    \begin{proof}
        Let $n$ be even, $i\in\overline{n}.$ It is easy to see that, if $i$ is odd, the set $\mathbb{YF}_{n,i}$ is empty, and if $i$ is even and $i=2j,$ we have $|\mathbb{YF}_{n,i}|=\binom{n/2+j}{2j},$ since $\mathbb{YF}_{n,i}$ contains all vertices with $n/2+j$ digits, of which $2j$ are ones. Thus, for sufficiently large $n$,
    \begin{gather*}
        \sum_{i=0}^{\lceil c\cdot\ln{n} \rceil-1}|\mathbb{YF}_{n,i}|\le  \sum_{j=0}^{\left\lfloor\frac{ c\cdot\ln{n} }{2}\right\rfloor}\binom{\frac{n}{2}+j}{2j}\le c\cdot \ln{n}\cdot\left(\frac{n}{2}+\frac{c\cdot\ln{n}}{2}\right)^{c\cdot\ln{n}} \le n^{c\cdot\ln{n}}.
    \end{gather*}        
    The case of odd $n$ is entirely analogous.
        \end{proof}

Lemma \ref{enl1} and Proposition \ref{enp4} immediately imply
\begin{entheorem}\label{ent2}
Let $c>1/(2W(1/e))\approx 1,79556.$ Then the sequence
$\{f(n)\}_{n=1}^\infty={\exp {(c\cdot\ln^2{n})}},$ is covering.
\end{entheorem}

\section{Non-covering sequences of the form $\exp{(c\cdot\ln^2{n})}$}

In this section, we need the following notation:

Let $a\in(1,2),$ $n\in\mathbb{N}_0.$ Then
    \begin{gather*}
        M_a(n):=\sum_{v\in\mathbb{YF}_n}\mu^a_P(v)=\sum_{v\in\mathbb{YF}_n}\frac{D^{2a}(v)}{(n!)^a}.
    \end{gather*}
\begin{enProp}\label{enp5}
    Let $a\in(1,2).$ Then $M_a(0)=M_a(1)=1,$ and, for $n\ge 2$, we have
    \begin{gather*}
        M_a(n)=\frac{1}{n^a}\cdot M_a(n-1)+\frac{(n-1)^a}{n^a}\cdot M_a(n-2).
    \end{gather*}
\end{enProp}
\begin{proof}
The equality $M_a(0)=M_a(1)=1$ is clear. For $n\ge 2$, we use \eqref{enfd}:
\begin{align*}
M_a(n)=&\sum_{v\in\mathbb{YF}_{n-1}}\frac{D^{2a}(1v)}{(n!)^a}+\sum_{v\in\mathbb{YF}_{n-2}}\frac{D^{2a}(2v)}{(n!)^a}\\
=&\frac{1}{n^a}\cdot\sum_{v\in\mathbb{YF}_{n-1}}\frac{D^{2a}(v)}{((n-1)!)^a}+\frac{(n-1)^a}{n^a}\cdot\sum_{v\in\mathbb{YF}_{n-2}}\frac{D^{2a}(v)}{((n-2)!)^a}.
\end{align*}
The proposition follows.
\end{proof}

\begin{enLemma}\label{enl2}
    Let $a\in(1,2),$ $n\in\mathbb{N}_0.$ Then
    \begin{gather*} 
    M_a(n)\le\frac{\exp{\frac{a}{a-1}}}{n^{a/2}}.
    \end{gather*}    
\end{enLemma}
\begin{proof}
    First, we prove that, for all $a\in(1,2),$ $n\in\mathbb{N}_0$, we have
    \begin{gather}\label{enp6} 
    M_a(n)\le\frac{\prod_{m=1}^n(1+{1}/{m^a})}{(n+1)^{a/2}}.
    \end{gather}    
    For a fixed $a\in(1,2)$, we prove \eqref{enp6} by induction on $n.$ The base case $n\in\{0,1\}$ is obvious. We prove the induction step for $n\ge 2.$

    Apply Proposition \ref{enp5}, the induction hypothesis, and the generalized Bernoulli inequality:
    \begin{align*}      
    M_a(n)=&\frac{1}{n^a}\cdot M_a(n-1)+\frac{(n-1)^a}{n^a}\cdot M_a(n-2)\\
\le&\frac{1}{n^a}\cdot\frac{\prod_{m=1}^{n-1}(1+1/m^a)}{n^{a/2}}+\frac{(n-1)^a}{n^a}\cdot\frac{\prod_{m=1}^{n-2}(1+1/m^a)}{(n-1)^{a/2}}\\
\le& \frac{\prod_{m=1}^{n-1}(1+1/m^a)}{(n+1)^{a/2}}\cdot\left( \frac{1}{n^a}\cdot\frac{(n+1)^{a/2}}{n^{a/2}}+\frac{(n^2-1)^{a/2}}{n^a}\right)\\
\le& \frac{\prod_{m=1}^{n-1}(1+1/m^a)}{(n+1)^{a/2}}\cdot\left(\frac{1}{n^a}\cdot\left(1+\frac{a}{2n}\right)+\left(1-\frac{a}{2n^2}\right)\right).
\end{align*}    
The induction step follows.

Now we apply \eqref{enp6} to prove the lemma:
\begin{gather*}
    M_a(n)\le \frac{\prod_{m=1}^n(1+1/m^a)}{(n+1)^{a/2}} 
    \le\frac{\exp{{\sum_{m=1}^\infty {1/m^a}}}}{n^{a/2}}\le \frac{\exp{\left(1+\int_{1}^\infty{dm}/{m^a}\right)}}{n^{a/2}}=\frac{\exp{\frac{a}{a-1}}}{n^{a/2}}.
\end{gather*}    
\end{proof}

\begin{entheorem}\label{ent3}
The sequence $\{f(n)\}_{n=1}^\infty={\exp {(1/16\cdot\ln^2{n})}},$ is not covering.
\end{entheorem}
\begin{proof}
Note that, for any $a\in(1,2),$ $n\in\mathbb{N}_0,$ by Lemma \ref{enl2} and Jensen's inequality for convex functions, we have
    \begin{gather*}
        \frac{\exp{\frac{a}{a-1}}}{n^{a/2}}\ge M_a(n)\ge\sum_{v\in A_f(n)}\mu_P^a(v)\ge\left(\frac{\mu_P(A_f(n))}{|A_f(n)|}\right)^a\cdot {|A_f(n)|}\ge \frac{\mu_P^a(A_f(n))}{\exp{(\frac{1}{16}\ln^2{n}\cdot(a-1) )}},
    \end{gather*}
and hence,
\begin{gather}\label{enfinal}
    \mu_P(A_f(n))\le \frac{\exp{\frac{1}{a-1}}}{\exp{\left(\frac{1}{2}\ln{n}\right)}}\cdot \exp{\left(\frac{1}{16}\ln^2{n}\cdot\frac{a-1}{a}\right)}.
\end{gather}

For each natural number $n\ge 55$, substitute into \eqref{enfinal} the value $a=1+4/\ln{n}:$
\begin{gather*}
    \mu_P(A_f(n))\le \frac{\exp{(\frac{1}{4}\ln{n})}}{\exp{\left(\frac{1}{2}\ln{n}\right)}}\cdot \exp{\frac{\frac{1}{4}\ln^2{n}}{\ln{n}+4}}=\exp{\left(-\frac{\ln{n}}{\ln{n}+4}\right)}.
\end{gather*}
This inequality shows that, for sufficiently large $n$, the value $\mu_P(A_f(n))$ is less than, for example, $1/2.$ The theorem follows immediately.
\end{proof}

\section{Acknowledgments}
I am very grateful to my academic adviser, Fedor Vladimirovich Petrov, for proposing the problem, for his patience in checking the proof and editing the paper, and for his moral support throughout the entire period of work. 

This work was supported by Saint Petersburg State University, project No. 159547043.

\section{AI disclosure}

The translation from Russian original was assisted by 
ChatGPT-6.1 Sol. Mathematical content is due to the author.

\end{document}